\documentclass[12pt,a4paper]{amsart}
\usepackage{a4wide}
\usepackage{amsmath, amssymb, amsfonts,enumerate}
\usepackage[all]{xy}
\usepackage{amscd}
\usepackage{comment}

\newtheorem{theorem}{Theorem}[section]

\newtheorem{corollary}[theorem]{Corollary}
\newtheorem{proposition}[theorem]{Proposition}

 \theoremstyle{definition}
 
 \newtheorem{remark}[theorem]{Remark}

 \newtheorem{example}[theorem]{Example}

\numberwithin{equation}{section}
\newcommand {\N}{\mathbb{N}} 
\newcommand {\Z}{\mathbb{Z}} 
\newcommand {\R}{\mathbb{R}} 

\newcommand{\T}{\mathbb{T}}

\newcommand{\FF}{\mathcal{F}}

\newcommand{\PP}{\mathcal{P}}

\newcommand{\UU}{\mathcal{U}}
\newcommand{\VV}{\mathcal{V}}

\DeclareMathOperator{\card}{card}
\DeclareMathOperator{\Per}{Per}

\DeclareMathOperator{\Fix}{Fix}

\DeclareMathOperator{\Id}{Id}

\DeclareMathOperator{\ent}{ent}

\begin{document}
\title[Expansive actions of countable amenable groups]{Expansive actions of countable amenable groups, homoclinic pairs, and the Myhill property}
\author{Tullio Ceccherini-Silberstein}
\address{Dipartimento di Ingegneria, Universit\`a del Sannio, C.so
Garibaldi 107, 82100 Benevento, Italy}
\email{tceccher@mat.uniroma3.it}
\author{Michel Coornaert}
\address{Institut de Recherche Math\'ematique Avanc\'ee,
UMR 7501,                                             Universit\'e  de Strasbourg et CNRS,
                                                 7 rue Ren\'e-Descartes,
                                               67000 Strasbourg, France}
\email{coornaert@math.unistra.fr}
\begin{abstract}
Let $X$ be a compact metrizable space equipped with a continuous action of a countable amenable group $G$.
Suppose  that the dynamical system $(X,G)$  is expansive and is the quotient 
by a uniformly bounded-to-one factor map of  a strongly irreducible subshift.
Let $\tau \colon X \to X$ be  a continuous map commuting with the action of $G$.
We prove that if there is no pair of distinct $G$-homoclinic points in $X$ having the same image under $\tau$ then $\tau$ is surjective.
\end{abstract}
\date{\today}
\maketitle

\section{Introduction}

Let $G$ be a countable group and  $A$ a finite set.
Consider the set $A^G$ consisting of all maps $u \colon G \to A$.
This set is called the set of \emph{configurations} over the group $G$ and the \emph{alphabet} 
$A$.
Equip $A^G$ with its \emph{prodiscrete topology}, that is,  the 
topology of pointwise convergence. 
The \emph{shift action} of $G$  on $A^G$ 
is the continuous action   
defined by $gu(h) := u(g^{-1}h)$ for all $g,h \in G$ and $u \in A^G$.
A \emph{cellular automaton} over $A^G$ is a continuous map $\tau \colon A^G \to A^G$ that is $G$-equivariant, i.e., satisfies $\tau(g u) = g \tau(u)$ for all $g \in G$ 
and $u \in A^G$.   
Two configurations  $u, v  \in A^G$ are said to be \emph{almost equal} if 
they coincide outside of a finite subset of $G$.
A cellular automaton $\tau \colon A^G \to A^G$ is said to be  \emph{pre-injective}
if there exist no distinct configurations $u, v \in A^G$ that are almost equal and satisfy 
$\tau(u) = \tau(v)$.
The celebrated  \emph{Garden of Eden theorem}, originally established by Moore and Myhill in the early 1960s,  states that a cellular automaton 
over the group $\Z$ of integers  is surjective if and only if it is pre-injective.
Actually, the implication surjective $\Rightarrow$ pre-injective was first established by Moore~\cite{moore} and,
shortly after,
 Myhill~\cite{myhill} proved the converse implication.
The Moore-Myhill Garden of Eden theorem was extended to all amenable groups
in~\cite{ceccherini}.
It follows from a result of Bartholdi~\cite{bartholdi}
that if a group $G$ is non-amenable then there exist cellular automata over $G$ that are surjective but not pre-injective.
Thus, the Garden of Eden theorem yields a characterization of amenability for groups. 
 \par
The goal of the present paper is to extend the Myhill implication in the Garden of Eden theorem to  certain   dynamical systems
$(X,G)$, consisting of a compact metrizable space $X$ equipped with a continuous action of a countable  amenable group $G$.
Our motivations come from Gromov (cf. \cite[Section 8.H]{gromov-esav}) 
who mentions the possibility of extending the Garden of Eden theorem 
to a ``suitable class of hyperbolic dynamical systems".
\par
Before stating our main result, let us briefly recall some additional definitions
(see Section~\ref{sec:background} for more details).
A closed $G$-invariant subset of $A^G$ is called a \emph{subshift}.
  A subshift $\Sigma \subset A^G$ is said to be \emph{strongly irreducible} if there is a finite subset 
 $\Delta \subset G$ satisfying the following property:
 if $\Omega_1$ and $\Omega_2$   are finite subsets of $G$
such that   there exists no element $g \in \Delta$ such that 
the right-translate of $\Omega_1$ by $g$ meets  
$\Omega_2$,  then, given any two configurations $u_1,u_2 \in \Sigma$, there exists  a configuration $u \in \Sigma$ which coincides with $u_1$ on $\Omega_1$ and with $u_2$ on $\Omega_2$.
\par
Let $(X,G)$ be a dynamical system consisting of a compact metrizable space $X$ equipped with  a continuous action of the  group $G$.
Two points $x,y \in X$ are said to be \emph{homoclinic} with respect to the action of $G$ on $X$, or more briefly $G$-\emph{homoclinic}, if for every $\varepsilon > 0$ there exists a finite subset $F \subset G$ such that 
$d(g x, g y) < \varepsilon$ for all $g \in G \setminus F$ (here $d$ denotes any metric on $X$ that is compatible with the topology).  
We say that a continuous $G$-equivariant map $\tau \colon X \to X$  is
\emph{pre-injective} with respect to the action of $G$ if there is no pair of distinct $G$-homoclinic points  in $X$ having the same image under $\tau$.
When $X = A^G$ and $G$ acts on $X$ by the shift, this definition is equivalent to the one given above (see Proposition~\ref{p:homoclinic-configurations}).
We say that the dynamicall system $(X,G)$ has the \emph{Myhill property} if
every $G$-equivariant continuous map   $\tau \colon X \to X$   
that is pre-injective with respect to the action of $G$ 
is surjective.
\par
Our main result  is the following. 

\begin{theorem}
\label{t:myhill-hyp}
Let $X$ be a  compact metrizable space equipped  with a continuous action of a countable amenable group $G$. 
Suppose that  the dynamical system $(X,G)$ is expansive 
and that there exist a finite set $A$, a strongly irreducible subshift 
$\Sigma \subset A^G$, and a uniformly bounded-to-one factor map 
$\theta \colon \Sigma \to X$.
Then the dynamical system $(X,G)$ has the Myhill property.
\end{theorem}

As the shift action on every  subshift $\Sigma \subset A^G$ is expansive,
we deduce from Theorem~\ref{t:myhill-hyp}
(by taking $\theta := \Id_\Sigma$, the identity map on $\Sigma$)
that if $G$ is a countable amenable group and $A$ is a finite set, then every strongly
irreducible subshift $\Sigma \subset A^G$ has the Myhill property.
This last result had been already established  by the 
authors~\cite[Theorem 1.1]{csc-myhill}.
In the particular case when $\Sigma = A^G$ is the full subshift, it yields the Myhill  implication in the Garden of Eden theorem for cellular automata over amenable groups 
established in~\cite{ceccherini}.  
\par
Theorem~\ref{t:myhill-hyp} had been previously obtained by the authors \cite[Theorem~1.1]{csc-goehyp}
when $G = \Z$ and 
$\Sigma \subset A^\Z$ is a topologically mixing subshift of finite type.
Actually, it is well known that a subshift of finite type $\Sigma \subset A^\Z$ is strongly irreducible if and only if it is topologically mixing. On the other hand,
there exist strongly irreducible subshifts $\Sigma \subset A^\Z$ that are not of finite type and  even not sofic (see e.g.~\cite{csc-periodic}), so that
the above result for $G = \Z$ is stronger than
the one in~\cite{csc-goehyp}.  
\par
According to Gromov \cite[Section~5]{gromov-hyperbolic-manifolds-groups-actions},
a dynamical system $(X,G)$ is  \emph{hyperbolic}
(or \emph{finitely presented}~\cite{fried})
 if it is expansive and a factor of some subshift of finite type.
 Thus, if the dynamical system $(X,G)$ satisfies the hypotheses 
 of Theorem~\ref{t:myhill-hyp}
 with $\Sigma$ of finite type, then $(X,G)$ is hyperbolic in the sense of Gromov.
 However, as already mentioned above, there are strongly irreducible subshifts over $\Z$ that are not sofic and hence not finitely presented. 
Consequently, there are dynamical systems $(X,G)$ satisfying all the hypotheses of Theorem~\ref{t:myhill-hyp}
without being hyperbolic in the sense of Gromov.  
\par 
Note also that there exist dynamical systems $(X,G)$
satisfying all the hypotheses of Theorem~\ref{t:myhill-hyp}
that do not have the \emph{Moore property}, i.e., admitting  continuous 
surjective $G$-equivariant maps 
$\tau \colon X \to X$ that are not pre-injective
(cf.~\cite{csc-goehyp}).
In the case $G = \Z$,
an example of such a dynamical system is provided by the even subshift
$X \subset \{0,1\}^\Z$ (see \cite[Section~3]{fiorenzi-sofic}).
\par
Following a terminology introduced by Gottschalk~\cite{gottschalk} 
(cf.~\cite{gromov-esav}),
let us say that a dynamical system $(X,G)$ is \emph{surjunctive} if every
injective $G$-equivariant continuous map $\tau \colon X \to X$ is surjective.

\begin{corollary}
\label{c:ds-is-surj}
Every dynamical system  $(X,G)$ satisfying the hypotheses of Theorem~\ref{t:myhill-hyp}
is surjunctive.
\end{corollary}

\begin{proof}
Injectivity trivially implies pre-injectivity.
\end{proof} 

The paper is organized as follows.
Section~\ref{sec:background} contains basic definitions and preliminary results.
The proof of Theorem~\ref{t:myhill-hyp}
is given in Section~\ref{sec:proof-main-result}.
It relies on an entropic argument.
In Section~\ref{sec:existence-homoclinic-pairs}, we prove that, in  a non-trivial dynamical system $(X,G)$ that is the quotient by a finite-to-one factor map of a strongly irreducible subshift, all $G$-homoclinicity classes are infinite
(Corollary~\ref{c:homoclinic-class-infinite}).
This applies in particular to any non-trivial dynamical system satisfying the hypotheses of Theorem~\ref{t:myhill-hyp} 
but arequires neither expansiveness of the system nor amenability of the acting group.
On the other hand, extending a result previously obtained by 
Schmidt~\cite{schmidt-pacific} in the case  
$G = \Z^d$, we show that if $G$ is a countable amenable group and $(X,G)$ is a dynamical system with positive topological entropy that is the quotient by a unfiormly bounded-to-one  factor map of a subshift of finite type,
then, for every integer $n \geq 1$, 
there is a $G$-homoclinicity class in $X$ containing more than $n$ points
(Corollary~\ref{c:X-G-factor-sft}).
Furthermore, generalizing a result also previously obtained 
by Schmidt~\cite{schmidt-pacific}
for $G =  \Z^d$, we prove that if $A$ is a finite set, $G$ a countable residually finite amenable group, and $\Sigma \subset A^G$
a subshift of finite type with zero topological entropy whose periodic configurations are dense,  then every $G$-homoclinicity class in $\Sigma$ is trivial.
In Section~\ref{sec:surj-density},  we prove the surjunctivity of expansive dynamical systems containing a dense set of periodic points.
This last result is  well known for subshifts and does not require amenability of the acting group.
The final section contains a description  of some examples of expansive dynamical systems that show the importance of the hypotheses in the above  results.
       
\section{Background and notation}
\label{sec:background}

In this section, we set up notation and collect basic facts that will be used in the sequel.
Some proofs are given for convenience.

\subsection{Dynamical systems}
The cardinality of a set $X$ is denoted  $\card(X)$.
A set $X$ is \emph{countable} if $\card(X) = \card(\N)$. 
Here $\N$ denotes the set of 
non-negative integers.
An action of a group $G$ on a set $X$ is a map $\alpha \colon G \times X \to X$ satisfying
$\alpha(g_1,\alpha(g_2,x)) = \alpha(g_1 g_2, x)$ and $\alpha(1_G, x) = x$ for all $g_1, g_2 \in G$ and $x \in X$, where $1_G$ denotes the identity element of $G$.
When the action $\alpha$ is clear from the context, we shall write $g x$ 
instead of $\alpha(g,x)$.
If a group $G$ acts on two sets $X$ and $Y$, a map $\varphi \colon X \to Y$ is called $G$-\emph{equivariant} if one has 
$\varphi(g x ) = g \varphi( x )$ for all $g \in G$ and $x \in X$.
An action of a group $G$ on a topological space $X$ is said to be \emph{continuous} if the permutation of  $X$  given by $x \mapsto g x$ is continuous for each $g \in G$.
\par   
Throughout  this paper,
by a \emph{dynamical system}, we shall mean a triple $(X,G,\alpha)$, where 
 $X$ is a compact metrizable space, $G$ is a countable group, 
   and  $\alpha$ is a continuous  action of $G$ on $X$.
   If there is no risk of confusion about the action, 
   we shall write $(X,G)$, or even sometimes simply $X$,  instead of $(X,G,\alpha)$. 
We shall denote by $d$   a metric  on $X$ that is compatible with the topology.
\par
Given a dynamical system $(X,G)$,
the \emph{orbit} of a point $x \in X$ is the set 
$\{g x  : g \in G\} \subset X$. The point $x$ is called \emph{periodic} if its orbit is finite.
The set $\Per(X,G) \subset X$ consisting of all periodic points of the dynamical system 
$(X,G)$  satisfies
\begin{equation}
\label{e:periodic-points-X-G}
\Per(X,G) = \bigcup_{H } \Fix(H),
\end{equation} 
where $H$ runs over all finite-index subgroups of $G$ 
and $\Fix(H)$ is the closed subset of $X$ consisting of all the points in $X$ that are fixed by $H$.
\par
A subset $Y \subset X$ is said to be \emph{invariant} if the orbit of every point of $Y$ is contained in $Y$.
If $Y \subset X$ is an invariant subset
then the action of $G$ on $X$ induces, by restriction, a continuous action of $G$ on $Y$.
\par
One says that the dynamical system  $(X,G)$ is \emph{expansive} if there exists a constant 
$\delta  > 0$ such that, for every pair of distinct points $x,y \in X$, there exists an element
$g = g(x,y)  \in G$ such that
$d(g x,g y) \geq  \delta$.
Such a constant $\delta$ is  called an \emph{expansiveness constant} for $(X,G,d)$.  
The fact that $(X,G)$ is expansive or not does not depend on the choice of the metric $d$.
Actually, the dynamical system $(X,G)$ is expansive if and only if there is a neighborhood $W \subset X \times X$ of the diagonal such that,
for every pair of distinct points $x,y \in X$,
there exists an element 
$g = g(x,y)  \in G$ such that
$(g x, g y) \notin W$.
Such a set $W$ is then called an
\emph{expansiveness set} for $(X,G)$. 
\par
One says that the dynamical system $(X,G)$  is \emph{topologically mixing} if, for any pair of nonempty open subsets $U$ and $V$ of $X$, there exists a finite subset $F \subset G$ such that $U  \cap gV \neq \varnothing$ for all $g \in G \setminus F$.
\par
Suppose that the group $G$ acts continuously on two compact metrizable spaces $X$  and $\widetilde{X}$.
\par
One says that the dynamical systems $(X,G)$ and 
$(\widetilde{X},G)$ are \emph{topologically conjugate} 
if there exists a $G$-equivariant homeomorphism 
$h \colon \widetilde{X} \to X$.
\par
One says that the dynamical system $(X,G)$ is a \emph{factor} of the dynamical system 
$(\widetilde{X},G)$ if there exists a $G$-equivariant continuous surjective map
$\theta \colon \widetilde{X} \to X$. 
Such a map $\theta$ is then called a \emph{factor map}.
A factor map $\theta \colon \widetilde{X} \to X$ is said to be \emph{finite-to-one} if the pre-image set 
$\theta^{-1}(x)$ is finite for each $x \in X$.
A finite-to-one factor map is said to be \emph{uniformly bounded-to-one} if there is an integer $K \geq 1$ such that
$\card(\theta^{-1}(x)) \leq K$ for all  $x \in X$. 

\subsection{Homoclinicity}
Let $X$ be a compact metrizable space equipped with a continuous action of a countable group $G$.
Let $d$ be a metric on $X$ compatible with the topology.
Two points $x, y \in X$ are called \emph{homoclinic} with respect to the action of $G$, 
 or more briefly  $G$-\emph{homoclinic}, if for every $\varepsilon > 0$, there is a finite subset $F \subset G$ such that 
$d(g x,g y) < \varepsilon $ for all $g \in G \setminus F$.
Homoclinicity defines an equivalence relation on $X$. 
By compactness of $X$, this equivalence relation  is independent  of the choice of the metric $d$. Its equivalence classes are called the $G$-\emph{homoclinicity classes} 
of $X$.

\begin{proposition}
\label{p:homoclinic-factor}
Let $\widetilde{X} $ and $X$ 
be  compact metrizable spaces, each equipped with a continuous action of a countable group $G$. 
Suppose that the dynamical system $(X,G)$ is a factor of the dynamical system $(\widetilde{X},G)$ and let
$\theta \colon \widetilde{X} \to X$ be a factor map.
Let $\widetilde{x}$ and $\widetilde{y}$ be points in $\widetilde{X}$ that are $G$-homoclinic.
 Then the points $x := \theta(\widetilde{x})$ and $y := \theta(\widetilde{y})$ are $G$-homoclinic.
\end{proposition}

\begin{proof}
Let $\widetilde{d}$ (resp. $d$) be  a metric on $\widetilde{X}$ (resp. $X$) that is compatible with the topology.
\par
Let $\varepsilon > 0$.
By compactness,  $\theta$ is uniformly continuous.
Therefore, 
there exists $\eta > 0$ such that
\begin{equation}
\label{e:pi-eta-unif}
\widetilde{d}(\widetilde{x}_1,\widetilde{x}_2) < \eta  \Rightarrow d(\theta(\widetilde{x}_1),\theta(\widetilde{x}_2)) <  \varepsilon
\end{equation}
for all $\widetilde{x}_1,\widetilde{x}_2 \in \widetilde{X}$.
Since the points $\widetilde{x}$ and $\widetilde{y}$ are $G$-homoclinic,
there is a finite subset $F \subset G$ such that
$\widetilde{d}(g \widetilde{x}, g \widetilde{y}) < \eta$ for all $g \in G \setminus F$.
It follows that, for all $g \in G \setminus F$,
\begin{align*}
d(g x,g y) 
&= d(g\theta(\widetilde{x}),g \theta(\widetilde{y})) \\ 
&= d(\theta( g \widetilde{x}),\theta(g \widetilde{y})) && \text{(since $\theta$ is $G$-equivariant)} \\
&< \varepsilon && \text{(by \eqref{e:pi-eta-unif})}. 
\end{align*}
This shows that the points $x$ and $y$ are $G$-homoclinic.
\end{proof}

\subsection{Amenability}
A countable group $G$
 is called \emph{amenable} if there exists a sequence 
$ (F_n)_{n \in \N}$ of non-empty finite subsets of $G$  satisfying, for all $g \in G$,
\begin{equation}
\label{e:folner-net}
\lim_{n \to \infty}  \frac{\card( F_n \setminus F_n g)}{\card( F_n )} = 0. 
\end{equation}
  Such a sequence $(F_n)_{n \in \N}$ is  called a \emph{F\o lner sequence} for $G$.
\par
There are many other equivalent definitions of amenability in the literature
(see e.g. the monographs
\cite{greenleaf}, \cite{paterson},  \cite{csc-book} and the references therein). 
 \par
All locally finite groups, all solvable groups (and therefore all abelian groups), and all finitely generated groups of subexponential growth are amenable.
The free group of rank $2$ provides an example of a non-amenable group.
As the class of amenable groups is closed under taking subgroups, it follows that  if a group $G$ contains a nonabelian free subgroup then $G$ is not amenable.
\par
We shall  use the following well known results about F\o lner sequences.

\begin{proposition}
\label{p:folner-times-finite-set}
Let $G$ be a countable amenable group and let $(F_n)_{n \in \N}$ be a F\o lner sequence for $G$.
Let $E$ be a non-empty finite subset of $G$.
Then the following hold:
\begin{enumerate}[\rm (i)]
\item
one has
  \begin{equation}
  \label{e:lim-folner-infinite}
  \lim_{n \to \infty} \card(F_n) = \infty ;
  \end{equation} 
\item
the sequence $(E F_n)_{n \in \N}$ is a F\o lner sequence for $G$;
\item
one has
\begin{equation}
\label{p:folner-times-finite-set-iii}
\lim_{n \to \infty} \frac{\card( F_n E \setminus F_n)}{\card(F_n)} = 0 ;
\end{equation} 
\item
one has
\begin{equation}
\label{p:folner-times-finite-set-iv}
\lim_{n \to \infty} \frac{\card( F_n E)}{\card(F_n)} = 1.
\end{equation}
\end{enumerate}
\end{proposition}

\begin{proof}
Let $K \in \N$. 
Since $G$ is infinite, we can find a finite subset $R \subset G$ with $\card(R) = K^2$.
It follows from~\eqref{e:folner-net} that, for each $g \in R$, there exists $N(g) \in \N$ such that 
$\card(F_n \setminus F_n g) < \card(F_n)$ for all $n \geq N(g)$.
This implies that $F_n$ meets $F_n g$ and hence  $g \in F_n^{-1} F_n$ 
for all $n \geq N(g)$.
If we take $N := \max_{g \in R} N(g)$, we then get $R \subset F_n^{-1} F_n$
for all $n \geq N$.
We deduce that
$$
K^2 = \card(R) \leq \card(F_n^{-1} F_n) \leq \card(F_n^{-1})\card(F_n) = (\card(F_n))^2
$$
and hence $\card(F_n) \geq K$ for all $n \geq N$. 
This shows~(i).
\par
For all $g \in G$,
we have  the inclusion
\begin{equation*}
E F_n  \setminus  E F_n  g 
\subset
 E (F_n \setminus F_n g).
\end{equation*}
This implies
\begin{equation*}
\card(E F_n  \setminus E F_n  g)
\leq \card(E)  \card ( F_n  \setminus  F_n g ).  
 \end{equation*}
As $\card(E F_n) \geq \card(F_n)$, this gives us
$$
\frac{\card(E F_n  \setminus E F_n  g)}{\card(E F_n)} 
\leq \card(E) \frac{\card( F_n \setminus F_n g)}{\card( F_n )} 
$$
and hence
$$
\lim_{n \to \infty} \frac{\card(E F_n  \setminus E F_n  g)}{\card(E F_n)} = 0
$$
by using~\eqref{e:folner-net}. This shows (ii).
\par
To prove (iii), we first observe that
$$
 F_n E \setminus F_n = 
\bigcup_{e \in E}(F_n e \setminus F_n) 
$$
so that 
\begin{align*}
\card(F_n E \setminus F_n) 
&= 
\card \left(\bigcup_{e \in E}(F_n e \setminus F_n)\right)\\
&\leq   \sum_{e \in E} \card(F_n e \setminus F_n)\\
&=  \sum_{e \in E} \card(F_n \setminus F_n e) 
&& \text{(since  $\card(F_n) = \card(F_n e$)}.
\end{align*}
and hence
\begin{equation}
\label{e:bound-for-boundary}
\frac{\card(F_n E \setminus F_n)}{\card(F_n)} \leq
\sum_{e \in E} \frac{\card(F_n \setminus F_n e)}{\card(F_n)}. 
\end{equation}
This shows (iii) since the right-hand side of~\eqref{e:bound-for-boundary} tends to $0$ as $n \to \infty$ by~\eqref{e:folner-net}.
\par
To prove (iv), we observe that
$$
F_n E \subset 
F_n \cup (F_n E \setminus F_n) 
$$
so that 
\begin{equation*}
\card(F_n E)
\leq \card(F_n \cup (F_n E \setminus F_n)) 
\leq  \card(F_n) + \card(F_n E \setminus F_n).
\end{equation*}
As $\card(F_n E) \geq \card(F_n)$, we deduce that
$$
1 \leq
\frac{\card( F_n E)}{\card(F_n)}
\leq 
1 + \frac{\card(F_nE \setminus F_n)}{\card(F_n)}. 
$$
This gives us  (iv) by using (iii). 
\end{proof}

\subsection{Topological entropy}
Consider a dynamical system $(X,G)$, where $X$ is a compact metrizable space and $G$ is a countable amenable group acting continuously on $X$.
\par
Let $\UU = (U_i)_{i \in I}$ be an open cover of $X$.
The \emph{cardinality} of $\UU$ is by definition the cardinality of its index set $I$.
One says that an open cover $\VV = (V_j)_{j \in J}$ is a \emph{subcover} of $\UU$ if 
$J \subset I$ and $V_j = U_j$ for all $j \in J$.
Since $X$ is compact,  $\UU$ admits a subcover with finite cardinality.
We denote by  $N(\UU)$  
 the smallest integer 
$n \geq 0$ such that $\UU$ admits a subcover with cardinality $n$.
For $g \in G$, we define the open cover $g \UU$ by $g \UU := (g U_i)_{i \in I}$.
\par
Let $\VV = (V_j)_{j \in J}$ be an open cover of $X$.
  The \emph{join}  of the open covers
 $\UU$ and $\VV$  
is the open cover $\UU \vee \VV$ of $x$ defined by   
 $\UU \vee \VV := (U_i \cap V_j)_{(i,j) \in I \times J}$.
 \par
 Given an open cover $\UU$ of $X$ and a non-empty finite subset  $F \subset G$,
 we define the open cover $\UU_F$  by
 $$
 \UU_F := \bigvee_{g \in F} g^{-1} \UU.
 $$
 Now let $\FF = (F_n)_{n \in \N}$ be a F\o lner sequence for $G$.
 It follows from the Ornstein-Weiss lemma as stated in
 in \cite[Theorem 6.1]{lindenstrauss-weiss} (see also  \cite{cscf-fekete} and the references therein) 
 that the limit
 $$
 h(\UU,X,G) := \lim_{n \to \infty} \frac{\log N(\UU_{F_n})}{\card(F_n)}
 $$
 exists, is finite, and does not depend on the choice of the 
 F\o lner sequence $\FF$ for $G$.
 \par
 The \emph{topological entropy} of the dynamical system $(X,G)$ is the 
 quantity $0 \leq h_{top}(X,G) \leq \infty$ given by
 $$
 h_{top}(X,G) := \sup_\UU  h(\UU,X,G),
$$
where $\UU$ runs over all open covers of $X$.
\par
The above definition of topological entropy was introduced by Adler, Konheim, and McAndrew~\cite{adler-entropy} in the case when  $G = \Z$. 
Let us now briefly review the metric approach to topological entropy that was developed 
independently by Bowen~\cite{bowen-periodic-TAMS-1971} and Dinaburg~\cite{dinaburg}
 for $G = \Z$.
\par 
Let $d$ be a metric on $X$ compatible with the topology.
\par
Given a non-empty finite subset  $F \subset G$,
we define the metric $d_F$ on $X$ by
\begin{equation}
\label{e:def-metric-bowen}
d_F(x,y) := \max_{g \in F} d(g x,g y) \quad \text{for all } x,y \in X.
\end{equation}
The metric $d_F$ is also compatible with the topology on $X$. 
Given  a real number $\varepsilon > 0$, 
one says that a subset $Z \subset X$ is an 
$(F,\varepsilon,X,G,d)$-\emph{spanning subset} of $X$ 
if for every $x \in X$, there exists $z \in Z$ such that
$d_F( x, z) < \varepsilon$.
By compactness, $X$ always contains a finite $(F,\varepsilon,X,G,d)$-spanning subset.  
Let $S(F,\varepsilon,X,G,d)$ denote the minimal cardinality of an $(F,\varepsilon,X,G,d)$-spanning subset $Z \subset X$.
\par
Let $\FF = (F_n)_{n \in \N}$ be a F\o lner sequence for $G$.
We define   $h(\FF,\varepsilon,X,G,d) $ by
\begin{equation}
\label{e:h-S-folner}
h(\FF,\varepsilon,X,G,d) := \limsup_{n \to \infty} \frac{\log S(F_n,\varepsilon,X,G,d)}{\card(F_n)}. 
\end{equation}
Observe that the map $\varepsilon \mapsto h(\FF,\varepsilon,X,G,d)$ is 
non-increasing.
This implies that
$$
\sup_{\varepsilon > 0} h(\FF,\varepsilon,X,G,d) = \lim_{\varepsilon \to 0} h(\FF,\varepsilon,X,G,d).
$$

\begin{proposition}
\label{p:top-entropy-metric-int}
With the notation above, one has
\begin{equation}
\label{e:def-top-entropy}
h_{top}(X,G) = 
\sup_{\varepsilon > 0} h(\FF,\varepsilon,X,G,d)
= \lim_{\varepsilon \to 0} h(\FF,\varepsilon,X,G,d).
\end{equation}
\end{proposition}

\begin{proof}
To simplify notation, let us set
 $$
 H = \sup_{\varepsilon > 0} h(\FF,\varepsilon,X,G,d).
 $$
 Let $\UU$ be an open cover of  $X$.
Choose some Lebesgue number $\lambda > 0$ for $\UU$ with respect to the metric $d$.
We recall that this means that every open $d$-ball of $X$ with radius $\lambda$ is contained in some member of $\UU$.
Observe that,
for every non-empty finite subset  $F \subset G$,
the open cover $\UU_F$ admits $\lambda$ as a Lebesgue number with respect to the metric $d_F$. 
We deduce  that
$$
N(\UU_F) \leq S(F,\lambda,X,G,d).
$$
It follows  that
$$
h(\UU,X,G) \leq h(\FF,\lambda,X,G,d) \leq H
$$
and hence
$$
h_{top}(X,G) = \sup_\UU  h(\UU,X,G) \leq H.
$$
It remains only to prove that $H \leq h_{top}(X,G)$.
Let $\varepsilon > 0$ and consider the open cover $\UU$ of $X$ consisting of all open balls of radius $\varepsilon$.
Then, for every non-empty finite subset $F \subset G$, we clearly have
$$
S(F,\varepsilon,X,G,d) \leq N(\UU_F). 
$$
This gives us
$$
h(\FF,\varepsilon,X,G,d) \leq h(\UU,X,G) \leq h_{top}(X,G) 
$$
and hence
$$
H = \sup_{\varepsilon > 0} h(\FF,\varepsilon,X,G,d) \leq h_{top}(X,G).
$$
This completes the proof of~\eqref{e:def-top-entropy}.
\end{proof}

One of the key properties of topological entropy is that it never increases when taking factors. More precisely, we have the following. 

\begin{proposition}
\label{p:entropy-factor}
Let $G$ be a countable amenable group and let  $X$ (resp.~$\widetilde{X}$) be a compact metrizable space equipped with a continuous action of $G$.
Suppose that  the dynamical system $(X,G)$ is a factor of the dynamical system  
$(\widetilde{X},G)$. 
Then one has 
$h_{top}(X,G) \leq h_{top}(\widetilde{X},G)$.
\end{proposition}

\begin{proof}
Let $\theta \colon \widetilde{X} \to X$ be a factor map
and let $\FF = (F_n)_{n \in \N}$ be a F\o lner sequence for $G$.
Let $\varepsilon > 0$ and
take $\widetilde{d}$,  $d$, and $\eta$ as in the proof of
Proposition~\ref{p:homoclinic-factor}.
\par 
Suppose that $F$ is a non-empty finite subset of $G$.
If $\widetilde{Z}$ is an
$(F,\eta,\widetilde{X},G,\widetilde{d})$-spanning subset of $\widetilde{X}$,
then $Z = \theta(\widetilde{Z})$
is  an $(F,\varepsilon,X,G,d)$-spanning subset of $X$.
It follows that
$$
S(F,\varepsilon,X,G,d) \leq S(F,\eta,\widetilde{X},G,\widetilde{d}). 
$$
This gives us
$$
S(\FF,\varepsilon,X,G,d) 
\leq S(\FF,\eta,\widetilde{X},G,\widetilde{d})
\leq h_{top}(\widetilde{X},G)
$$ 
and hence
$$
h_{top}(X,G) = \sup_{\varepsilon > 0} S(\FF,\varepsilon,X,G,d)
\leq h_{top}(\widetilde{X},G).
$$
\end{proof}

\subsection{Shifts and subshifts}
Let $G$ be a countable group and
let $A$ be a finite set.
We equip $A^G$   with its prodiscrete topology.
This is the product topology obtained by taking the discrete topology on each factor $A$ of $A^G$. The space $A^G$, being the product of countably many finite discrete spaces, is metrizable, compact, and totally disconnected. 
It is homeomorphic to the Cantor set as soon as  $A$ has more than one element.
\par
Let us choose a non-decreasing sequence $(R_n)_{n \in \N}$ of finite subsets of $G$ such that
$R_0 = \varnothing$, $R_1 = \{1_G\}$,  and $G = \bigcup_{n \in \N} R_n$.
Then the  metric $\rho$ on $A^G$, defined by
\begin{equation}
\label{e:def-metric-shift}
\rho(u,v) :=
\frac{1}{k},
\end{equation}
where 
$$
k := \sup \{n \geq 1 \text{ such that } u(g) = v(g) \text{  for all } g \in R_{n - 1} \}
$$
 (with the usual convention $1/\infty = 0$), is compatible with the topology.
\par
The  $G$-shift action on $A^G$ 
is  expansive and admits  $\delta = 1$  
as an expansiveness constant with respect to the metric $\rho$.
Actually, the open set
$$
W := \{(u,v) \in A^G \times A^G : u(1_G) \not= v(1_G) \} \subset A^G \times A^G
$$
is an expansiveness set for $(A^G,G)$.

\begin{proposition}
\label{p:homoclinic-configurations}
Let $G$ be a countable group and
let $A$ be a finite set. 
Then two configurations $u, v \in A^G$ are $G$-homoclinic  if and only if they are almost equal.
\end{proposition}

\begin{proof}
Suppose first that  $u$ and $v$ are almost equal.
Then there exists a finite subset $F \subset G$
such that $u(g) = v(g)$ for all $g \in G \setminus F$.
Let $\varepsilon > 0$ and choose a positive integer $k$ such that $1/(k+1) < \varepsilon$.
If $g \in G$ is such that $g^{-1} F$ does not meet the set $R_k$,
that is, $g \notin F R_{k}^{-1}$, then $u(g) = v(g)$.
As the set $F R_{k}^{-1}$ is finite, this shows that  $u$ and $v$ are $G$-homoclinic.
\par
Conversely, suppose that $u$ and $v$ are $G$-homoclinic.
Then we can find a finite subset $F$ of $G$  such that $\rho(g u,g v) \leq 1/2$ for all
$g \in G \setminus F$.
This implies that $u$ and $v$ coincide outside of $F^{-1}$.
Consequently, the configurations $u$ and $v$ are almost equal.
\end{proof}

A  $G$-invariant closed subset $\Sigma \subset A^G$ is called a
\emph{subshift}.
\par
Let $\Delta$ be a finite subset of $G$.
Given  subsets $\Omega_1$ and $\Omega_2$ of $G$, one says that $\Omega_1$ is
$\Delta$-\emph{apart}
from $\Omega_2$ if one has $\Omega_1 \Delta \cap \Omega_2 = \varnothing$, i.e.,
there is no element $g \in \Delta$ such that  
the right-translate by $g$ of $\Omega_1$ meets $\Omega_2$.
A subshift $X \subset A^G$ is called $\Delta$-\emph{irreducible} if it satisfies the following condition:
given  any two finite subsets $\Omega_1, \Omega_2 \subset G$  such that $\Omega_1$ is 
$\Delta$-apart
from $\Omega_2$ and any two configurations  $u_1, u_2\in \Sigma$, there exists a configuration $u \in \Sigma$ which coincides
with $u_1$ on $\Omega_1$ and 
with $u_2$ on $\Omega_2$. 
A subshift $\Sigma \subset A^G$ is called \emph{strongly irreducible} if there exists a finite subset 
$\Delta \subset G$  such that $\Sigma$ is $\Delta$-irreducible.
\par
We shall use the following result.

\begin{proposition}
\label{p:strongly-implies-stab}
Let $G$ be a countable group and let $A$ be a finite set.
Let $\Delta$ be a finite subset of $G$ and let $\Sigma \subset A^G$ be a 
$\Delta$-irreducible subshift.
Suppose that 
$\Omega_1$ and $\Omega_2$   are  (possibly infinite) subsets of $G$
such that $\Omega_1$ is  $\Delta$-apart from $\Omega_2$. 
Then, given any two configurations $u_1$ and $ u_2$ in  $\Sigma$, there is a configuration 
$u$ in  $\Sigma$ 
which coincides with $u_1$ on $\Omega_1$ and with $u_2$ on $\Omega_2$.
\end{proposition}

\begin{proof}
This is a particular case of \cite[Lemma 4.6]{csc-myhill}.
\end{proof}

A subshift $\Sigma \subset A^G$ is said to be of \emph{finite type} if there exist a finite subset 
$\Omega \subset G$ and a subset $\PP \subset A^\Omega$ such 
that 
\begin{equation}
\label{e:def-sft}
\Sigma = \{u \in A^G : (g u)\vert_\Omega \in \PP \text{  for all  } g \in G\}.
\end{equation}
\par
A subshift $\Sigma \subset A^G$ is called \emph{sofic} if there exist a finite set $B$ and a subshift of finite type $\Psi \subset B^G$ such that $\Sigma$ is a factor of $\Psi$. 

\subsection{Topological entropy of subshifts}
Let $G$ be a countable group and $A$ a finite set.
Given a subset $F \subset G$, le $\pi_F \colon A^G \to A^F$ denote the canonical projection map, i.e., the map defined by $\pi_F(u) = u\vert_F$ for all $u \in A^G$, where 
$u\vert_F \in A^F$ denotes the restriction of $u$ to $F$. 

\begin{proposition}
\label{p:top-entropy-subshift}
Let $G$ be a countable amenable group and $A$ a finite set.
Let $\Sigma \subset A^G$ be a subshift.
Let $\FF = (F_n)_{n \in \N}$ be a 
F\o lner sequence for $G$.
Then the topological entropy of $\Sigma$  satisfies
\begin{equation}
\label{e:entropy-sub-shift}
h_{top}(\Sigma,G) = \lim_{n \to \infty} \frac{\log \card(\pi_{F_n^{-1}}(\Sigma))}{\card(F_n)}.
\end{equation}
\end{proposition}

\begin{proof}
We first observe that the existence of the limit in the right-hand side 
of~\eqref{e:entropy-sub-shift}
is again an immediate consequence of the Ornstein-Weiss lemma that was mentioned above.
Indeed,  let $\PP_{fin}(G)$ denote the set of all non-empty finite subsets of $G$.
Then, one immediately  checks that the map $\eta \colon \PP_{fin}(G) \to \N$ defined by
$\eta(F) := \card(\pi_{F^{-1}}(\Sigma))$ is
right-invariant
(i.e., $\eta(Fg) = \eta(F)$ for all $F \in \PP_{fin}(G)$ and $g \in G$) 
non-decreasing
(i.e., $\eta(F) \leq \eta(F')$ for all $F,F' \in \PP_{fin}(G)$ with $F \subset F'$),
and submultiplicative in the sense that  $\eta(F \cup F') \leq \eta(F) \eta(F')$ 
for all disjoint $F,F' \in \PP_{fin}(G)$.
Thus, the map $F \mapsto \log \eta(F)$ satisfies the hypotheses needed to apply the Ornstein-Weiss lemma \cite[Theorem~6.1]{lindenstrauss-weiss}. Consequently, the limit
$$
\lim_{n \to \infty} \frac{\log \card(\pi_{F_n^{-1}}(\Sigma))}{\card(F_n)}
$$
exists and does not depend on the choice of the F\o lner sequence $\FF$.  
\par
Now let us assume that the F\o lner sequence $\FF$ is \emph{symmetric}, i.e., satisfies
$F_n = F_n^{-1}$ for all $n \in \N$
(the existence of such a F\o lner sequence follows from \cite[Corollary~5.3]{namioka}).
Note that this implies in particular that
$$
\lim_{n \to \infty} \frac{\log \card(\pi_{F_n^{-1}}(\Sigma))}{\card(F_n)}
 =
 \lim_{n \to \infty} \frac{\log \card(\pi_{F_n}(\Sigma))}{\card(F_n)}.
$$
Equip $\Sigma$ with the metric $\rho$ defined by~\eqref{e:def-metric-shift}.
Let $F$ be a non-empty finite subset of $G$ and fix an integer $k \geq 1$.
It immediately follows from~\eqref{e:def-metric-shift} and~\eqref{e:def-metric-bowen} that
if two configurations $u,v \in \Sigma$ coincide on
$F^{-1} R_k$ then $\rho(u,v) < 1/k$
while $\rho(u,v) \geq 1/k$ if $u$ and $v$ do not coincide on $F^{-1} R_k$. 
This gives us
$$
S(F,1/k,\Sigma,G,\rho) = \card(\pi_{F^{-1} R_k}(\Sigma))
$$
and hence  
\begin{align}
\label{e:h-folner-for-ent-subs}
h(\FF,1/k,\Sigma,G,\rho) &= 
\limsup_{n \to \infty} \frac{\log \card(\pi_{F_n^{-1} R_k}(\Sigma))}{\card(F_n)}
&& \text{(by~\eqref{e:h-S-folner})} \\
&= \limsup_{n \to \infty} \frac{\log \card(\pi_{F_n R_k}(\Sigma))}{\card(F_n)}
&& \text{(since $\FF$ is symmetric)}. \notag
\end{align}
Now observe that $1_G \in R_k$ so that
$F_n \subset F_n R_k$.
We deduce that
$$
\card(\pi_{F_n }(\Sigma)) \leq
\card(\pi_{F_n R_k}(\Sigma)) \leq
\card(\pi_{F_n }(\Sigma)) \times \card\left(A^{F_n R_k \setminus F_n}\right).
$$
This yields$$
\frac{\log \card(\pi_{F_n }(\Sigma))}{\card(F_n)} \leq
\frac{\log \card(\pi_{F_n R_k}(\Sigma))}{\card(F_n)} \leq
\frac{\log \card(\pi_{F_n }(\Sigma))}{\card(F_n)}
+ \frac{\card(F_n R_k \setminus F_n)}{\card(F_n)}
$$
for all $n \in \N$.
Letting $n \to \infty$, this gives us
\begin{equation}
\lim_{n \to \infty} \frac{\log \card(\pi_{F_n R_k}(\Sigma))}{\card(F_n)} =
\lim_{n \to \infty} \frac{\log \card(\pi_{F_n}(\Sigma))}{\card(F_n)}
\end{equation}
since
$$
\lim_{n \to \infty} \frac{\card(F_n R_k \setminus F_n)}{\card(F_n)} = 0
$$
by~\eqref{p:folner-times-finite-set-iii}.
\par
By applying~\eqref{e:h-folner-for-ent-subs}, we then get
$$
h(\FF,1/k,\Sigma,G,\rho) =
\lim_{n \to \infty} \frac{\log \card(\pi_{F_n}(\Sigma))}{\card(F_n)}
$$
for all $k \geq 1$.
As
$$
h_{top}(\Sigma,G) = \lim_{k \to \infty}
h(\FF,1/k,\Sigma,G,\rho)
$$
by Proposition~\ref{p:top-entropy-metric-int},
this shows~\eqref{e:entropy-sub-shift}. 
\end{proof}

\begin{remark}
If the F\o lner sequence $\FF$ is symmetric, then
formula~\eqref{e:entropy-sub-shift} shows that $h_{top}(\Sigma,G)$
coincides with the entropy
 $\ent_\FF(\Sigma)$ defined  in~\cite{csc-myhill}.
 \end{remark}

We shall also need the following result about topological entropy of strongly irreducible subshifts.

\begin{proposition}
\label{p:entropy-increasing}
Let $G$ be a countable  amenable group, $A$ a finite set, and $\Sigma \subset A^G$ a strongly irreducible subshift.
Suppose that $\Psi \subset A^G$ is a subshift that is strictly contained in $\Sigma$.
Then one has $h_{top}(\Psi,G) < h_{top}(\Sigma,G) $. 
 \end{proposition}

\begin{proof}
See \cite[Proposition~4.2]{csc-myhill}.
\end{proof}

\section{Proof of the main result}
\label{sec:proof-main-result}

In this section, we present the proof of Theorem~\ref{t:myhill-hyp}.
Recall that we are given a countable amenable group $G$ acting continuously on a compact metrizable space $X$, a finite set $A$, a strongly irreducible subshift 
$\Sigma \subset A^G$,  
a uniformly bounded-to-one factor map $\theta \colon \Sigma \to X$,
and a $G$-equivariant continuous map $\tau \colon X \to X$.
We want to show that if $\tau$ is pre-injective with respect to the action of $G$ on $X$ then $\tau$ is surjective.
We shall proceed by contradiction.
So assume that $\tau$ is not surjective.
Then $Y := \tau(X)$ 
is a closed $G$-invariant proper subset of $X$.
As $\theta$ is a factor map,
it follows that   $\Psi := \theta^{-1}(Y)$  is a proper subshift of $\Sigma$.
Since $\Sigma$ is strongly irreducible, we deduce from 
Proposition~\ref{p:entropy-increasing} that
\begin{equation}
\label{e:ent-Psi-less-Sigma}
h_{top}(\Psi,G) < h_{top}(\Sigma,G).
\end{equation}
Let us choose a metric $d$ on $X$ that is compatible with the topology and 
let $\delta > 0$ be an expansiveness constant for $(X,G,d)$.
As $\Sigma$ is compact, the composite map $\tau \circ \theta \colon \Sigma \to X$ is uniformly continuous. 
Consequently, there is a finite subset $L \subset G$  such that
\begin{equation}
\label{e:tau-pi-unif-cont}
v\vert_L = w\vert_L  \Rightarrow d(\tau(\theta(v)),\tau(\theta(w))) < \delta
\end{equation}
for all  $v,w \in \Sigma$.
\par
Let $\FF = (F_n)_{n \in \N}$ be a symmetric F\o lner sequence for $G$
(as already mentioned in the proof of Proposition~\ref{p:top-entropy-subshift}, the existence of such a F\o lner sequence follows from \cite[Corollary~5.3]{namioka}).
Let $ \Delta$ be a finite subset of $G$ such that $\Sigma$ is $\Delta$-irreducible.
Up to replacing $\Delta$ by $\Delta \cup \Delta^{-1}$, 
we can assume that $\Delta$ is symmetric.
\par
Now let us fix some configuration $u \in \Sigma$ and consider, for each  $n \in \N$, 
the subset  $\Phi(n) \subset \Sigma$
consisting of all configurations in $\Sigma$ that coincide with $u$ outside of $F_n \Delta$, that is, 
  \[
\Phi(n) := \{v \in \Sigma: v\vert_{G \setminus F_n \Delta} = u\vert_{G \setminus F_n \Delta} \} \subset \Sigma.
\]

The set  $\Omega_1:= F_n$ is $\Delta$-apart from the set  $\Omega_2 := G \setminus F_n\Delta$. 
Therefore,
it follows from Proposition~\ref{p:strongly-implies-stab} that,
given any configuration $w \in \Sigma$, there exists 
a configuration $v \in \Phi(n)$ such that
$v$ coincides with $w$ on $F_n$.
This implies  that
\begin{equation}
\label{e:card-Phi-n-big}
\card(\Phi(n)) \geq \card(\pi_{F_n}(\Sigma))
\end{equation}
for all $n \in \N$.
\par
As $\theta$ is uniformly bounded-to-one, there is an integer $K \geq 1$ such that every point in $X$ has at most $K$ pre-images under $\theta$.
By using~\eqref{e:card-Phi-n-big}, 
this gives us
\begin{equation}
\label{e:pi-Phi-n-big}
\card(\theta(\Phi(n))) \geq K^{-1}\card(\Phi(n)) \geq K^{-1} \card(\pi_{F_n}(\Sigma))
\end{equation}
for all $n \in \N$.
\par
If $g \in G \setminus L \Delta F_{n} $, 
then 
$$
g^{-1} \in G \setminus F_n^{-1} \Delta^{-1} L^{-1} = G \setminus F_n \Delta L^{-1}
$$
 and hence
$g^{-1}L \subset G \setminus F_n\Delta$.
Thus, if $v, w \in \Phi(n)$
and $g \in G \setminus L \Delta F_{n} $,
then $(g v)\vert_L = (g u)\vert_L = (g w)\vert_L$.
By applying~\eqref{e:tau-pi-unif-cont},
we deduce that all $v,w \in \Phi(n)$ satisfy
\begin{equation}
\label{e:d-tau-teta-gv-gw}
d(\tau(\theta(gv)),\tau(\theta(gw))) < \delta
\quad
\text{for all  } g \in G \setminus L \Delta F_{n.}
\end{equation}

Let now $x,y \in \theta(\Phi(n))$ and
choose $v, w \in \Phi(n)$ such that $x = \theta(v)$ and $y = \theta(w)$.
We have that
\begin{align*}
d(g\tau(x),g\tau(y)) 
&= d(g \tau(\theta(v)), g \tau(\theta(w))) \\
&= d(\tau(g \theta(v)),\tau(g \theta(w))) 
&& \text{(since $\tau$ is  $G$-equivariant)} \\
&= d(\tau(\theta(gv)),\tau(\theta(gw))) 
&&\text{(since $\theta$ is $G$-equivariant).}
\end{align*}
Therefore, by using~\eqref{e:d-tau-teta-gv-gw}, we get
\begin{equation}
\label{e:d-f-k-tau-less-1}
d(g\tau(x),g\tau(y)) < \delta 
\quad
\text{for all  } g \in G \setminus L \Delta F_{n.}
\end{equation}
\par
We now observe that
\begin{align*}
\label{e:ent-f-L-n-M}
 h_{top}(\Sigma,G) 
&= \lim_{n \to \infty} \frac{\log \card(\pi_{F_n^{-1}}(\Sigma))}{\card(F_n)} 
&& \text{(by Proposition~\ref{p:top-entropy-subshift})} \\
&= \lim_{n \to \infty} \frac{\log \card(\pi_{F_n}(\Sigma))}{\card(F_n)} 
&& \text{(since $\FF$ is symmetric)} \\ 
&= \lim_{n \to \infty} \frac{\log (K^{-1}\card(\pi_{F_n}(\Sigma))}{\card (F_n) }.
\end{align*}
On the other hand,
the sequence $\FF' = (L \Delta F_n)_{n \in \N}$ is a F\o lner sequence for $G$ 
by Proposition~\ref{p:folner-times-finite-set}.(ii).
Therefore, 
it follows from~\eqref{e:def-top-entropy} that
\begin{align*}
h_{top}(\Psi,G) 
&\geq h_{top}(Y,G) 
&& \text{(by Proposition~\ref{p:entropy-factor})} \\
&\geq h(\FF',\delta/2,Y,G,d) 
&&\text{(by Proposition \ref{p:top-entropy-metric-int})} \\ 
&= \limsup_{n \to \infty} \frac{\log S(L \Delta F_n,\delta/2,Y,G,d)}{\card(L \Delta F_n)} 
&& \text{(by \eqref{e:h-S-folner})} \\
& = \limsup_{n \to \infty} \frac{\log S(L \Delta F_n,\delta/2,Y,G,d)}{\card(F_n^{-1} \Delta^{-1} L^{-1})}
&& \text{(since $\card(L \Delta F_n) = \card(F_n^{-1} \Delta^{-1} L^{-1})$)} \\
& = \limsup_{n \to \infty} \frac{\log S(L \Delta F_n,\delta/2,Y,G,d)}{\card(F_n \Delta^{-1} L^{-1})}
&& \text{(since $F_n^{-1} = F_n$)} \\ 
& = \limsup_{n \to \infty} \frac{\log S(L \Delta F_n,\delta/2,Y,G,d)}{\card(F_n)}
&& \text{(by using~\eqref{p:folner-times-finite-set-iv}).}  
\end{align*}
As $h_{top}(\Sigma,G) > h_{top}(\Psi,G)$ by~\eqref{e:ent-Psi-less-Sigma},
we deduce from the two estimations above  that there exists  $n \in \N$ such that
\begin{equation}
\label{e:final-choice-n}
K^{-1} \card(\pi_{F_n}(\Sigma)) > S(L \Delta F_n,\delta/2,Y,G,d).
\end{equation}
Let $Z \subset Y$ be an $(L \Delta F_n,\delta/2,Y,G,d)$-spanning subset with minimal cardinality.
It follows from~\eqref{e:pi-Phi-n-big}  and~\eqref{e:final-choice-n} that
$$
\card(\theta(\Phi(n))) > S( L \Delta F_n, \delta/2,Y,G,d) = \card(Z).
$$
Therefore, by the pigeon-hole principle, there exist two distinct points $x,y \in \theta(\Phi(n))$ and a point $z \in Z$ such that
$$
d_{L \Delta F_{n}}(\tau(x),z)) < \delta/2  \quad \text{and} \quad d_{L \Delta F_{n}}(\tau(y),z)) < \delta/2.
$$
By applying the triangle inequality, this gives us
$$
d_{L \Delta F_{n}}(\tau(x),\tau(y))) < \delta,
$$
that is,
\begin{equation}
\label{e:d-f-k-tau-less-2}
d(g\tau(x),g\tau(y)) < \delta
\quad
\text{for all  } g \in L \Delta F_{n}.
\end{equation}
Since $\delta$ is an expansiveness constant for $(X,G,d)$, we deduce from
\eqref{e:d-f-k-tau-less-1} and \eqref{e:d-f-k-tau-less-2} that $\tau(x) = \tau(y)$.
The elements of $\Phi(n)$ are almost equal. Therefore they belong to the same homoclinicity class with respect to the $G$-shift  by Proposition~\ref{p:homoclinic-configurations}.
As $\theta$ is a factor map, it follows from Proposition~\ref{p:homoclinic-factor} that all points in $\theta(\Phi(n))$ are in the same homoclinicity class with respect to the action of 
$G$ on $X$.
Consequently, the points $x$ and $y$ are $G$-homoclinic.
As $x$ and $y$ are distinct 
and have the same image under $\tau$, this shows that $\tau$ is not 
pre-injective. 
This completes the proof of Theorem~\ref{t:myhill-hyp}.

\section{Existence of non-trivial homoclinicity classes}
\label{sec:existence-homoclinic-pairs}

If all $G$-homoclinicity classes of a dynamical system $(X,G)$ are trivial, i.e., each reduced to a single point, then, by definition,
every continuous $G$-equivariant map $\tau \colon X \to X$ is pre-injective.
This implies in particular that $(X,g)$ has the Moore property.
In this section, we present some results about the triviality or non-triviality of homoclinicity classes in certain subshifts and factors of subshifts. 

\begin{proposition}
\label{p:homoclinicity-class-strongly-irred}
Let $G$ be a countable  group and let $A$ be a finite set.
Suppose that $\Sigma \subset A^G$
is an infinite strongly irreducible subshift. 
Then every $G$-homoclinicity  class in  $\Sigma$ is infinite. 
\end{proposition}

\begin{proof}
Fix a  configuration $u \in \Sigma$ and let $\Phi \subset \Sigma$
denote the $G$-homoclinicity class of $u$. 
By Proposition~\ref{p:homoclinic-configurations},
the class $\Phi$ consists of all configurations 
$v \in \Sigma$ that are almost equal to $u$.
\par
Let $\Delta$ be a finite subset of $G$ such that $\Sigma$ is $\Delta$-irreducible.
For every finite subset $F \subset G$,
the set  $F$ is $\Delta$-apart from the set  $G \setminus F\Delta$. 
Therefore,
it follows from Proposition~\ref{p:strongly-implies-stab} that,
given any configuration $w \in \Sigma$, there exists 
a configuration $v \in \Sigma$ such that
$v$ coincides with $w$ on $F$
and  with $u$ on $G \setminus F \Delta$.
Observe that such a configuration $v$ is in $\Phi$ since the set $F \Delta$ is finite.
This implies that
\begin{equation}
\label{e:card-Ph-big}
\card(\Phi) \geq \card(\pi_{F}(\Sigma))
\end{equation}
for every finite subset $F \subset G$ (cf. the proof of Theorem~\ref{t:myhill-hyp} in Section~\ref{sec:proof-main-result}).
\par
On the other hand, as the subshift $\Sigma$ is infinite, for every $n \in \N$, there exists a finite subset $F \subset G$ such that
$\card(\pi_{F}(\Sigma)) \geq n$.   
We then deduce from~\eqref{e:card-Ph-big} that  the set $\Phi$ is infinite.
\end{proof}

The corollary below applies in particular to any non-trivial dynamical system $(X,G)$  satisfying the hypotheses of Theorem~\ref{t:myhill-hyp}.
Note that it  requires neither expansiveness of the system nor amenability of the acting group. 

\begin{corollary}
\label{c:homoclinic-class-infinite}
Let $X$ be an infinite   compact metrizable space equipped  with a continuous action of a countable  group $G$. 
Suppose that there exist a finite set $A$, a strongly irreducible subshift 
$\Sigma \subset A^G$, and a finite-to-one factor map 
$\theta \colon \Sigma \to X$.
Then every $G$-homoclinicity  class in $X$ is infinite. 
\end{corollary}

\begin{proof}
Every $G$-homoclinicity class in $\Sigma$ is infinite
by Proposition~\ref{p:homoclinicity-class-strongly-irred}.
As the factor map $\theta$ is finite-to-one,
we conclude that each $G$-homoclinicity class in $X$ is infinite 
by applying Proposition~\ref{p:homoclinic-factor}. 
\end{proof}

In \cite[Proposition 2.1]{schmidt-pacific}, Schmidt proved that if $\Sigma$ is a subshift of finite type 
over $\Z^d$ with positive topological entropy, 
then the $\Z^d$-homoclinicity relation on $\Sigma$ is non-trivial, i.e., ther exist two distinct configurations in $\Sigma$ that are $\Z^d$-homoclinic.
The following statement extends Schmidt's result in two directions.
Firstly, it applies to a subshift of finite type with positive topological entropy $\Sigma$ over an arbitrary countable amenable group $G$.
Secondly, it says that, for any integer $n \geq 2$, one can find $n$ distinct configurations in $\Sigma$ that are $G$-homoclinic.  

\begin{proposition}
\label{p:homoclinic-pair-sft}
Let $G$ be a countable amenable group and let $A$ be a finite set.
Suppose that $\Sigma \subset A^G$
is a subshift of finite type with $h_{top}(\Sigma,G) > 0$.
Then, for every integer $K \geq 1$, there exists a $G$-homoclinicity class in $\Sigma$
containing more than $K$ configurations.
\end{proposition}

\begin{proof}
We shall proceed by contradiction.
So let us assume that there is an integer $K \geq 1$
such that each $G$-homoclinicity class in $\Sigma$
contains at most $K$ configurations.
\par
As the subshift $\Sigma$ is of finite type, we can find a finite subset $\Omega \subset G$ and  $\PP \subset A^\Omega$ such that $\Sigma$ satisfies~\eqref{e:def-sft}.
Up to enlarging $\Omega$ if necessary, 
we can assume that $1_G \in \Omega$ and $\Omega = \Omega^{-1}$.
\par
Let $F \subset G$ be a finite subset.
Observe that
\begin{equation}
\label{e:inclusions-F-Omeg}
F \subset F \Omega \subset F \Omega^{2}
\end{equation}
since $1_G \in \Omega$.
Consider the finite subset $\partial F \subset G$ defined by
$$
\partial F := F \Omega^2 \setminus F
$$
and suppose that $u$ and $v$ are two configurations in $\Sigma$ such that
\begin{equation}
\label{e:u-v-coincide-boundary}
u\vert_{\partial F} = v\vert_{\partial F}.
\end{equation}
We claim that  the configuration $w \in A^G$ defined by
\begin{equation*}
w(g) = \begin{cases}
v(g) & \text{ if } g \in F\Omega\\
u(g) &  \text{ if } g \in G \setminus F\Omega
\end{cases}
\end{equation*}
also belongs to $\Sigma$.
\par
To see this, we first observe that it follows from~\eqref{e:inclusions-F-Omeg} 
and~\eqref{e:u-v-coincide-boundary}
that $w$ coincides with $u$ on $G \setminus F$ and with $v$ on $F \Omega^{2}$.
Now let $g \in G$.
\par
If $g \in F \Omega$, then $g \Omega \subset F \Omega^2$ so that
$$
(g^{-1}w)\vert_\Omega = (g^{-1}v)\vert_\Omega \in \PP
$$
since $v \in \Sigma$.
\par
On the other hand, if $g \in G \setminus F \Omega$, 
then $g \Omega \subset G \setminus F$ so that
$$
(g^{-1}w)\vert_\Omega = (g^{-1}u)\vert_\Omega \in \PP
$$
since $u \in \Sigma$.
\par
 Thus, we have that $(g^{-1}w)\vert_\Omega \in \PP$ for all $g \in G$.
This shows that $w \in \Sigma$.
\par
The configurations $u$ and $w$ are $G$-homoclinic since they coincide outside of $F$.
On the other hand, $w$ coincides with $v$ on $F$. 
As the $G$-homoclinicity class of $u$ contains at most $K$ configurations by our hypothesis,
we deduce that
\begin{equation}
\label{e:estimate}
\card(\pi_F(\Sigma)) \leq K \card(\pi_{\partial F}(\Sigma)) 
\end{equation}
for every finite subset $F \subset G$.
\par
Now let $(F_n)_{n \in \N}$ be a symmetric F\o lner sequence for $G$.
We then have
\begin{equation*}
h_{top}(\Sigma,G) = \lim_{n \to \infty} \frac{\log \card(\pi_{F_n}(\Sigma))}{\card(F_n)}
\end{equation*}
by Proposition~\ref{p:top-entropy-subshift}.
Since
\begin{align*}
\frac{\log \card(\pi_{F_n}(\Sigma))}{\card(F_n)} 
&\leq
\frac{\log (K\card(\pi_{\partial F_n}(\Sigma)))}{\card(F_n)}
&& \text{(by~\eqref{e:estimate} )} \\ 
&=
\frac{\log K}{\card(F_n)} + 
\frac{\log \card(\pi_{\partial F_n}(\Sigma))}{\card(F_n)} \\
&\leq
\frac{\log K}{\card(F_n)} + 
\frac{\log \card(A^{\partial F_n})}{\card(F_n)} \\
&=
\frac{\log K}{\card(F_n)} + 
\frac{\card(\partial F_n)}{\card(F_n)} \log \card(A) 
\end{align*}
 and
$$
\lim_{n \to \infty} \frac{\card(\partial F_n)}{\card(F_n)}
= \lim_{n \to \infty} \frac{\card(F_n \Omega^2 \setminus F_n)}{\card(F_n)} = 0
$$
by applying Proposition~\ref{p:folner-times-finite-set}.(iii),
we deduce that $h_{top}(\sigma,G) = 0$.
\end{proof}

The following example shows that a subshift $\Sigma \subset A^G$ that satisfies the hypotheses of Proposition~\ref{p:homoclinic-pair-sft}
may contain a trivial $G$-homoclinicity class.

\begin{example}
Let $G$ be a finitely generated amenable group (e.g. $G = \Z^d$)  and let $A$ be a finite set with more than two elements.
Fix an element $a_0 \in A$ and consider the subshift $\Sigma \subset A^G$
consisting of all configurations $u \in A^G$ such that either $u(g) = a_0$ for all $g \in G$
or $u(g) \not= a_0$ for all $g \in G$.
This is a subshift of finite type.
Indeed, if $S$ is a finite generating subset of $G$, then $\Sigma$ 
satisfies~\eqref{e:def-sft}
for $\Omega = \{1_G\} \cup S$
by taking as $\PP$ the set consisting of all maps 
$p \colon \Omega \to A$ such that either $p(g) = a_0$ for all $g \in \Omega$ or
$p(g) \not= a_0$ for all $g \in G$.
We clearly have $h_{top}(\Sigma,G) = \log(\card(A) - 1) > 0$.
On the other hand, the $G$-homoclinicity class of the constant configuration 
$u_0 \in \Sigma$, given by $u_0(g) = a_0$ for all $g \in G$,
is reduced to $u_0$. 
 \end{example}

As an immediate consequence of 
Proposition~\ref{p:homoclinic-pair-sft}, we get the following result.
Note that the expansiveness of the dynamical system $(X,G)$ is not required in the hypotheses.

\begin{corollary}
\label{c:X-G-factor-sft}
Let $X$ be an infinite   compact metrizable space equipped  with a continuous action of a countable amenable  group $G$ that satisfies $h_{top}(X,G) > 0$. 
Suppose that there exist a finite set $A$, a subshift of finite type 
$\Sigma \subset A^G$, and a uniformly bounded-to-one factor map 
$\theta \colon \Sigma \to X$.
Then, for every integer $n \geq 1$, there exists a $G$-homoclinicity class in $X$
containing more than $n$ points.
\end{corollary}

\begin{proof}
As $\theta$ is uniformly bounded-to-one, there is an integer $K \geq 1$ such that 
$\card(\theta^{-1}(x)) \leq K$ for all $x \in X$.
We have $h_{top}(\Sigma,G) > 0$ by Proposition~\ref{p:entropy-factor}.
Therefore, it follows from Proposition~\ref{p:homoclinic-pair-sft} that,
given an integer $n \geq 1$,
there is a $G$-homoclinicity class  $\Phi \subset \Sigma$ that contains more than 
$K n$ configurations.
Then its  image  $\theta(\Phi) \subset X$ has more than $n$ points and is contained in a $G$-homoclinicity class
since  $\theta$ preserves homoclinicity by Proposition~\ref{p:homoclinic-factor}.
\end{proof}

\begin{remark}
In \cite[Question 1.1]{chung-li-expansive},
Chung and Li asked whether every expansive system $(X,G)$, where $X$ is a compact metrizable space, $G$ is a countable amenable group, and
$h_{top}(x,G) > 0$, must contain two distinct points that are $G$-homoclinic.
This question is known to have an affirmative answer
for algebraic systems when $G = \Z^d$ \cite[Theorem~4.1]{lind-schmidt} and, mor generally,
when $G$ is polycyclic-by-finite
\cite[Theorem~1.1]{chung-li-expansive}.
We recall that the system $(X,G)$ is said to be \emph{algebraic} if $X$ is a compact abelian group and $G$ acts by automorphisms on $X$.
However, according to \cite{chung-li-expansive}, the answer to this question in the general case is unknown even for  $G = \Z$.  
\end{remark}

\begin{remark}
In \cite[Example 3.4]{lind-schmidt}, Lind and Schmidt
constructed examples of expansive algebraic actions of $\Z$ on the $n$-torus $\T^n$ with positive topological entropy for which 
every $\Z$-homoclinicity class is trivial.
\end{remark}


\par
Recall that a group $G$ is \emph{residually finite} provided that for every finite subset $F \subset G$ there exists
a finite index subgroup $H \subset G$ such that $\card(H \cap F) \leq 1$.
All finitely generated abelian groups (e.g.  $\Z^d$, for $d \geq1$) and, more generally, all finitely generated virtually nilpotent groups are residually finite.
\par
The following result yields a weak converse to Proposition~\ref{p:homoclinic-pair-sft} above.
In the particular case $G = \Z^d$, it reduces  to the second statement 
of Proposition~2.1 in \cite{schmidt-pacific}. 

\begin{proposition}
\label{p:reverse}
Let $A$ be a finite set, $G$ a countable amenable group, and $\Sigma \subset A^G$ a subshift of finite type.
Suppose that $G$ is residually finite and that the set  $\Per(\Sigma,G)$ is dense in $\Sigma$.
If $h_{top}(\Sigma,G) = 0$, then 
every $G$-homoclinicity class in $\Sigma$ is reduced to a single configuration.
\end{proposition}

\begin{proof}
We shall proceed by contradiction. 
So, let us suppose that $u_0$ and $u_1$ are two distinct
$G$-homoclinic configurations in $\Sigma$. 
Then we can find a finite subset $E \subset G$ such that
$u_0 \vert_{G \setminus E} = u_1 \vert_{G \setminus E}$.
Since $\Sigma$ is of finite type, we can find,
as in the proof of Proposition~\ref{p:homoclinic-pair-sft}, 
a finite symmetric subset $\Omega \subset G$, with
$1_G \in \Omega$, and $\PP \subset A^\Omega$ such that $\Sigma$ 
satisfies~\eqref{e:def-sft}.
\par
Since $\Per(\Sigma,G)$ is dense in $\Sigma$, there is a periodic configuration  
$v \in \Sigma$ such that
$v\vert_{E\Omega^2} = u_0\vert_{E\Omega^2}$. 
Let $H_v$ denote the stabilizer of $v$ in $G$ 
and observe that $[G:H_v] = \card(G v) < \infty$.
\par
We claim that we can find a finite index subgroup $H \subset H_v$ such that the left-translates $hE\Omega^2$, 
$h \in H$, are all pairwise disjoint. Indeed, since $G$ is residually finite, there exists  a finite index subgroup $H' \subset G$
such that $H' \cap E \Omega^2 (E \Omega^2)^{-1} = \{1_G\}$. 
Then  $H := H_v \cap H'$ is a finite index subgroup of
$G$.
Moreover, if $h_1,h_2 \in H$ are distinct,
then $h_1 h_2^{-1} \not= 1_G$, so that
$h_1 h_2^{-1} \notin E\Omega^2(E \Omega^2)^{-1}$
and hence $h_1 E \Omega^2$ does not meet $h_2 E \Omega^2$.
Thus, the subgroup $H$ has the required properties.
\par
Let $R \subset G$ be a complete set of representatives of the right cosets of $H$ in $G$.
Then $H$  satisfies $\bigcup_{h \in H} h R= G$.
As the sets $hE\Omega^2$, 
$h \in H$, are pairwise disjoint,
we deduce that $H$ is an $(E\Omega^2,R)$-tiling of $G$ in the sense
of~\cite[Section 5.6]{csc-book}.
\par
Consider now, for each $z \in \{0,1\}^H$, 
the configuration $w_z \in A^G$ defined by
\[
w_z(g) =
\begin{cases}
u_{z(h)}(h^{-1} g) & \text{ if } g \in h E \Omega^2  \text{ for some (necessarily unique)} h \in H \\
v(g) & \text{ otherwise}
\end{cases}
\]
for all $g \in G$.
Observe that,  since
  $u_0 \vert_{E \Omega^2 \setminus E} = u_1 \vert_{E \Omega^2 \setminus E}$, we have that
\begin{equation}
\label{e:u-v}
w_z \vert_{G \setminus \bigcup_{h \in H} h E} = v \vert_{G \setminus \bigcup_{h \in H} h E}.
\end{equation}
We claim that $w_z \in \Sigma$ for every $z \in \{0,1\}^H$. 
Indeed, let $g \in G$. We distinguish two cases.
\par
\emph{First case:} $g \in h E \Omega$ for some $h \in H$.
 Then  
$g\Omega \subset h E \Omega^2$ so that
\[
(g^{-1} w_z)\vert_\Omega = (g^{-1} h u_{z(h)})\vert_\Omega \in \PP,
\]
since $u_{z(h)} \in \Sigma$.
\par
\emph{Second case:}  $g \in G \setminus \bigcup_{h \in H} h E \Omega$.
 Then
$g\Omega  \subset G \setminus \bigcup_{h \in H} h E$ (here we use the fact that $\Omega$ is symmetric), 
so that, by virtue of~\eqref{e:u-v},
\[
(g^{-1}w_z)\vert_\Omega = (g^{-1} v)\vert_\Omega \in \PP
\]
since $v \in \Sigma$.
\par
Thus we have that $(g^{-1} w_z)\vert_\Omega \in \PP$ for all $g \in G$.
This shows that $w_z \in \Sigma$ for all $z \in \{0,1\}^H$.
\par
We are now in position to prove that $h_{top}(\Sigma, G) > 0$.
Let $(F_n)_{n \in \N}$ be a symmetric F\o lner sequence for $G$.
For each $n \in \N$, let $H_n$ be the finite subset of $G$ defined by 
\[
H_n := \{ h \in H: hE \subset F_n\}.
\]
By \cite[Proposition~5.6.4]{csc-book}  applied to the $(E\Omega^2,R)$-tiling $H$,
there exist a constant  $\alpha > 0$ and $n_0 \in \N$ such that
\begin{equation}
\label{e:alpha}
\card(H_n) \geq \alpha \card(F_n)
\end{equation}
for all $n \geq n_0$.
\par
As $u_0\vert_F \neq u_1\vert_F$,
it follows from the definition of $w_z$ and $H_n$ that
\[
\card(\pi_{F_n}(\{w_z : z \in \{0,1\}^H \})) \geq 2^{\card(H_n)}.
\]
Since $w_z \in \Sigma$ for all $z \in \{0,1\}^H$, it follows
from~\eqref{e:alpha} that
\[
\card(\pi_{F_n}(\Sigma)) \geq 2^{\card(H_n)} 
\geq 2^{\alpha \card(F_n)}
\]
for all $n \geq n_0$.
We then conclude that 
$$
h_{top}(\Sigma, G) = \lim_{n \to \infty} \frac{\log(\card(\pi_{F_n}(\Sigma)))}{\card(F_n)}
\geq \alpha \log 2 > 0
$$
by applying Proposition~\ref{p:top-entropy-subshift}.
\end{proof}
\section{Surjunctivity and density of periodic points}
\label{sec:surj-density}

The following surjunctivity result does not require amenability of the acting group
but only expansiveness of the dynamical system and density of periodic points.
It is well known for subshifts
(cf. \cite{gottschalk}, 
\cite[Section 4]{gromov-esav}, \cite[Theorem~8.2]{fiorenzi-periodic},   
\cite[Proposition~2.1]{csc-periodic})
and was previously obtained by the authors in \cite[Proposition~4.1]{csc-goehyp} 
for  $G = \Z$.  

\begin{proposition}
\label{p:ppd-implies-surjunctive}
Let $X$ be a compact metrizable space equipped with a continuous action of a countable group $G$. Suppose that  the dynamical system $(X,G)$ is expansive 
and that the set $\Per(X,G)$ of periodic points of $(X,G)$ is dense in $X$.
Then the dynamical system $(X,G)$ is surjunctive.
\end{proposition}

\begin{proof}
Suppose that $\tau \colon X \to X$ is an injective  $G$-equivariant continuous map.
Let $d$ be a metric on $X$ that is compatible with the topology and let $\delta > 0$ be an expansiveness constant for $(X,G,d)$. 
Fix a finite-index subgroup  $H$ of $G$ and let $R \subset G$ be a complete set of representatives of the left-cosets of $H$ in $G$, so that
every element $g \in G$ can be uniquely written in the form 
$g = rh$ with $r \in R$ and $h \in H$.
 Since the set $R$ is finite,  for $\varepsilon > 0$ small enough, 
all points $x,y \in X$ such that $d(x,y) < \varepsilon$ satisfy    $d(r x,r y) < \delta$
 for all $r \in R$.
By expansivity, this implies that all distinct points in $\Fix(H)$ are at least 
$\varepsilon$-apart.
As $X$ is compact, we deduce that $\Fix(H)$ is finite.
On the other hand, the $G$-equivariance of  $\tau$ implies that 
$\tau(\Fix(H)) \subset \Fix(H)$.
Since $\tau$ is injective, it follows that $\tau(\Fix(H)) = \Fix(H)$. 
We deduce that $\tau(\Per(X,G)) = \Per(X,G)$ by applying~\eqref{e:periodic-points-X-G}.
As $\Per(X,G)$ is dense in $X$, we conclude that $\tau(X) = X$.
This shows that $\tau$ is surjective. 
\end{proof}

\begin{corollary}
\label{c:factor-strong-irr-with-periodic-surj}
Let $X$ be a  compact metrizable space equipped  with a continuous action of a countable residually finite  group $G$. 
Suppose that  the dynamical system $(X,G)$ is expansive 
and that there exist a finite set $A$, a strongly irreducible subshift 
of finite type $\Sigma \subset A^G$ admitting a periodic configuration, 
such that $(X,G)$ is a factor of $(\Sigma,G)$.
Then the dynamical system $(X,G)$ is surjunctive.
\end{corollary}

\begin{proof}
By \cite[Theorem~1.1]{csc-periodic},
the set of periodic configurations is dense in $\Sigma$.
On the other hand, since $\theta$ is $G$-equivariant, 
every periodic configuration in $\Sigma$ is mapped by $\theta$ to a periodic point of $X$.
As $\theta$ is continuous and onto, we deduce  that $\Per(X,G)$ is dense in $X$. 
\end{proof}

\section{Some examples of expansive dynamical systems}

In this section, we describe expansive dynamical systems that may be used as counterexamples for showing the necessity of some of the hypotheses in our results.

\begin{example}
Let $G$ be a countable group.
Consider a dynamical system
 $(X,G)$, where $X$ is a finite discrete space with    
 $k \geq 2$   points
 and  $G$ fixes each point of $X$.
 Each $G$-homoclinicity class of $X$ is reduced to one  point and any map 
 $\tau \colon X \to X$ is continuous and $G$-equivariant.
As $X$ has more than one point,
we deduce that $(X,G)$ is surjunctive but does not have the Myhill property.
The   dynamical system $(X,G)$ satisfies the hypotheses of 
Proposition~\ref{p:ppd-implies-surjunctive}
 without having the Myhill property.
Observe that $(X,G)$  is topologically conjugate to the subshift 
$\Sigma \subset \{0,1,\dots,k - 1\}^G$ consisting of the $k$ constant configurations.
Note also that $\Sigma$ is of finite type if the group $G$ is finitely generated.
Indeed, if $S$ is a finite generating subset of $G$, then $\Sigma$ 
satisfies~\eqref{e:def-sft}
for $\Omega = \{1_G\} \cup S$
by taking as $\PP$ the set of constant maps from $\Omega$ to $\{0,1,\dots,k - 1\}$.
 \end{example}

\begin{example}
Let $G$ be a countable group equipped with the discrete topology and consider
its  one-point compactification $X = G \cup \{\infty\}$.
Observe that $X$ is compact and metrizable.
Actually, $X$ is homeomorphic to
   the subset of $\R$ consisting of $0$ and all the inverses of the positive integers. 
The action of  $G$ on itself by left-multiplication  continuously extends
to an action of $G$ on $X$ that fixes $\infty$.
This action is clearly expansive but not topologically mixing.
The points of $X$ are all in the same $G$-homoclinicity class.
On the other hand,
a map $\tau \colon X \to X$ is continuous and $G$-equivariant if and only if
either $\tau$ sends every point of $X$ to $\infty$
or there exists $g_0 \in G$ such that
$\tau(\infty) = \infty$ and $\tau(g) = g g_0$ for all $g \in G$.
We deduce that the dynamical system $(X,G)$ has the Myhill property and is therefore surjunctive.
However, when $G$ is amenable, 
$(X,G)$ does not satisfy the hypotheses of Theorem~\ref{t:myhill-hyp}.
Indeed, every strongly irreducible subshift is topologically mixing
(see e.g. \cite[Proposition~3.3]{csc-myhill})  and it is clear that any factor of a topologically mixing dynamical system is itself topologically mixing.
Clearly,  $(X,G)$ is topologically conjugate to the subshift $\Sigma \subset \{0,1\}^G$ consisting of all configurations $u \in \{0,1\}^G$ such that there is at most one element $g \in G$ such that $u(g) = 1$.
Consequently,  the dynamical system $(X,G)$ is finitely presented
if and only if the subshift $\Sigma$ is sofic.
It follows from  \cite[Corollary~4.4]{dahmani-yaman-symbolic} that $(X,G)$ is finitely presented  if 
$G$ is polyhyperbolic (e.g. polycyclic).
On the other hand, it is shown in~\cite[Proposition~2.3 and Theorem~2.11]{aubrin-effectiveness}
that if $G$ is a finitely generated and recursively presented group with undecidable word problem
then $(X,G)$ is not finitely presented. 
\end{example}

\begin{example}
Fix an integer $n \geq 2$  and let 
$G$ denote  the solvable (and hence amenable) 
Baumslag-Solitar group
given by the presentation
$G = \langle a,b : bab^{-1} = a^n \rangle$.
Consider the dynamical system $(X,G)$ where $X$ is the real projective line
$S^1 = \R \cup \{\infty\}$ and the action of $G$ on $X$ is the projective action defined by
$a x = x + 1$ and $b x = n x$ for all $x \in X =  \R \cup \{\infty\}$.
Observe that $\infty$ is the only point of $X$ fixed by $G$ and that the orbit of
 any point $x \in X \setminus \{\infty\}$ is dense in $X$.  
Note also that the action of $G$ on $X$ is expansive but not topologically mixing
and that every $G$-homoclinicity class is reduced to one single point.
On the other hand, it is easy to check that
a map $\tau \colon X \to X$ is continuous and $G$-equivariant if and only if either $\tau(x) = \infty$ for all $x \in X$ or $\tau$ is the identity map on $X$.
We deduce that $(X,G)$ is surjunctive but does not have the Myhill property.
Finally, let us note that    $(X,G)$ does not satisfy the conclusion of 
Proposition~\ref{p:homoclinic-pair-sft}.
 \end{example}

\begin{example}
Let $A$ be a finite set and $G$ a countable, amenable, residually finite  group.
Suppose that  $\Sigma \subset A^G$ is a subshift of finite type with 
$h_{top}(\Sigma,G) = 0$ and $\Per(\Sigma,G)$  dense in $\Sigma$.
Suppose also that $\Sigma$ contains a constant configuration $u_0$ but is not reduced to $u_0$.
Then all $G$-homoclinicity classes in $\Sigma$ are trivial by Proposition~\ref{p:reverse}.
Thus, the map $\tau \colon \Sigma \to \Sigma$ defined by $\tau(u) = u_0$ for all $u \in \Sigma$ is pre-injective.
As $\tau$ is continuous and $G$-equivariant but not surjective, this shows that  $\Sigma$ does not have the Myhill property.
Note however that $\Sigma$ is surjunctive by Proposition~\ref{p:ppd-implies-surjunctive}.
\par
An example of a subshift satisfying all these conditions is provided by the 
Ledrappier subshift
(cf. \cite[Section 3]{quas-trow}).  Recall that the \emph{Ledrappier subshift}  is the subshift
$\Sigma \subset A^G$, for $G = \Z^2$ and $A = \Z/2\Z$, consisting of all $x \colon \Z^2 \to \Z/2\Z$ satisfying
$$
x(m,n) + x(m + 1,n) + x(m,n+ 1) = 0
$$
for all $(m,n) \in \Z^2$.
As the Ledrappier subshift is topologically mixing, 
we deduce that Theorem~\ref{t:myhill-hyp}    becomes false if the hypothesis saying that the subshift $\Sigma$ is strongly irreducible is replaced by the condition that 
$\Sigma$ is topologically mixing and of finite type.
Note that $\Sigma$ is a compact abelian group and that $G$ acts on $\Sigma$ by group automorphisms, so that  $(\Sigma,G)$ is an expansive algebraic dynamical system
(cf.~\cite{schmidt-book}).  
\end{example}

\begin{example}
Weiss \cite[p. 358]{weiss-sgds} described a topologically mixing subshift of finite type containing a constant configuration over $\Z^2$ 
that  is not surjunctive and hence does not have the Myhill property.
This shows  that  
both Corollary~\ref{c:ds-is-surj} and Corollary~\ref{c:factor-strong-irr-with-periodic-surj}   become false if the hypothesis saying that the subshift $\Sigma$ is strongly irreducible is replaced by the condition that 
$\Sigma$ is topologically mixing and of finite type.
\end{example}

\begin{example}
Let $G$ be a group  containing a nonabelian free subgroup.
It is known that there is a finite set $A$ such that
$A^G$ does not have the Myhill property (cf.~\cite[Proposition~5.11.2]{csc-book}).
This shows in particular that
amenability of the group $G$ cannot be removed from the hypotheses of 
Theorem~\ref{t:myhill-hyp}.
However, there are non-amenable groups with no nonabelian free subgroups and we do not know if Theorem~\ref{t:myhill-hyp} becomes also false
when applied to  such groups.
An affirmative answer to this question would give a new characterization of amenablity for groups. 
\end{example}
 \bibliographystyle{siam}

\end{document}